\documentclass[10pt]{amsart}
\usepackage{array,youngtab}
\usepackage{color, graphicx, comment}
\usepackage{bbm}
\usepackage{tikz}
\usetikzlibrary{automata,positioning}

\usepackage{mathtools}

\allowdisplaybreaks

\numberwithin{equation}{section} \overfullrule 5pt
\newtheorem{Theorem}{Theorem}[section]
\newtheorem{Corollary}[Theorem]{Corollary}
\newtheorem{Conjecture}[Theorem]{Conjecture}
\newtheorem{Proposition}[Theorem]{Proposition}
\newtheorem{Lemma}[Theorem]{Lemma}

\theoremstyle{Definition}

\DeclareMathOperator{\Stiel}{Stiel}
\DeclareMathOperator{\Jac}{Jac}

\title[Automaticity of the Hankel determinants]{%
On the automaticity of sequences defined by continued fractions}
\date{August 13, 2019}

\author{Guo-Niu Han}
\address{Universit\'e de Strasbourg, CNRS, IRMA UMR 7501, F-67000 Strasbourg, France}
\email{guoniu.han@unistra.fr}

\author{Yining Hu}
\address{School of Mathematics and Statistics, 
Huazhong University of Science and Technology, Wuhan, PR China}
\email{huyining@hust.edu.cn}

\subjclass[2010]{11B85, 11J70, 11B50, 11Y65, 05A15}

\keywords{automatic sequence, continued fraction, Hankel determinant, Thue-Morse sequence, period-doubling sequence}

\begin{document}
\begin{abstract} 
	Continued fraction expansions of automatic sequences have been extensively
	studied during the last decades. The research interests are, on one hand, 
	in the degree or automaticity of the partial quotients following the seminal
paper of Baum and Sweet in 1976, and on the other hand, in calculating the
	Hankel determinants and irrationality exponent, as one can find in the works of 
	Allouche-Peyri\`ere-Wen-Wen, Bugeaud, and the first author.
The present paper is motivated by the converse problem: to study continued fractions whose coefficients form an automatic sequence.  We consider two such continued fractions defined by the Thue-Morse and period-doubling sequences respectively, and 
prove that they are congruent to algebraic series in $\mathbb{Z}[[x]]$
	modulo $4$. Consequently, the sequences of the coefficients 
	of the power series expansions of the two continued fractions
	modulo $4$
	are $2$-automatic.
Our approach is to first guess the explicit formulas of certain subsequences
	of $(P_n(x))$ and $(Q_n(x))$, where $P_n(x)/Q_n(x)$ is the canonical  
	representation of the truncated continued fractions, then prove these formulas
	by an intricate induction involving eight subsequences while exploiting 
	the relations between these subsequences.
\end{abstract}

\maketitle

\section{Introduction}\label{sec:Intro} 
	Continued fraction expansions of automatic sequences have been extensively
	studied during the last decades. 
	By the well-known Theorem of Christol \cite{Christol1980KMFR}, an element
	in $\mathbb{F}_q((1/x))$ is algebraic over $\mathbb{F}_q(x)$ if and only if
	its coefficients form a $q$-automatic sequence.
	In 1976, Baum and Sweet \cite{Baum1976S} proved that the continued fraction
	expansion of
	the unique solution $\varphi(x)$ in $\mathbb{F}_2((1/x))$ of the equation
	$$
  xf^3+f+x=0,
$$
  has partial quotients
	of bounded degree. They also gave examples of algebraic elements of degree
	greater than $2$ in $\mathbb{F}_q((1/x))$ whose continued fraction expansion
	have partial quotients of unbounded degree. In contrast, we do not know
	any algebraic real number of degree greater than $2$ that has bounded or 
	unbounded partial quotients. In \cite{Baum1977S}, for every even degree
	$d$, Baum and Sweet gave examples of elements in $\mathbb{F}_2((1/x))$, 
	algebraic of degree $d$, with bounded partial quotients. In \cite{Mills1986R},
	Mills and Robbins gave explicitly the continued fraction expansion of 
	$\varphi(x)$. They also gave examples, for each odd prime $p$, of algebraic 
	elements of degree greater than $2$ in $\mathbb{F}_p((1/x))$ whose partial 
	quotients are linear; this was largely generalized  by Lasjaunias and Yao 
	in \cite{Lasjaunias2015Y}.
	Allouche proved in \cite{Allouche1988} 
	that the sequences of partial quotients for the examples in \cite{Mills1986R} are
	automatic.  In contrast, Mkaouar proved that the sequence parital quotients
	of $\varphi(x)$, while being morphic, is not automatic. 
	In \cite{Lasjaunias2016Y} and \cite{Lasjaunias2017Y}, Lasjaunias and Yao
	considered the sequence
	of leading coefficients of the partial quotients instead of the partial
	quotients themselves and described several families of hyperquadratic 
	series for which the leading coefficients of the partial quotients form
	automatic sequences.

On the other hand, continued fraction expansions of automatic sequences have been studied for calculating the Hankel determinants and irrationality exponents \cite{han2015hankel,Han2016Adv,Fokkink2017etal,Badziahin2019}.  In 1998, Allouche, Peyri\`ere, Wen and Wen proved that all the Hankel determinants of the Thue-Morse sequence are nonzero [1]. This property allowed Bugeaud to prove that the irrationality exponents of the Thue-Morse-Mahler numbers are exactly 2 \cite{Bugeaud2011}. Since then, the Hankel determinants for several other automatic sequences, in particular, the paperfolding sequence, the Stern sequence and the period-doubling sequence, are studied by Coons, Vrbik, Guo, Wu, Wen, Bugeaud and Fu \cite{Coons2012V, Guo2014GWW,Fu2016H,Bugeaud2014H}. Using Jacobi continued fractions, the first author found a simple proof of APWW's result \cite{han2015hankel}. Finally, The Euler-Lagrange theorem says that the continued fraction expansion of a quadratic irrational number is ultimately periodic. The first author obtained similar result for quadratic power series on finite fields \cite{Han2016Adv}.  

The present paper is motivated by the converse problem: to study continued fractions whose coefficients form an automatic sequence.

We now give a brief introduction to automatic sequences. We refer the readers
to \cite[p. 185]{Allouche2003Sh} for more details.
Automatic sequences appear naturally in the study of various domains of
mathematics and theoretical computer science.
One of the equivalent definitions of automatic sequences is the following:
for an integer $k\geq 2$, a sequence $(u_n)_{n\geq 0}$ is said to be 
{\it $k$-automatic} if its {\it $k$-kernel}, defined as 
 $$\{ (u(k^d n+j))_{n\geq 0} \mid d\in \mathbb{N},\, 0\leq j\leq k^d-1 \},$$
is finite. Thus, if we denote by $\Lambda_j$ the Cartier operators \cite[p. 376]{Allouche2003Sh} that
maps $\sum_{n=0}^{\infty} a_n x^n$ to $\sum_{n=0}^{\infty} a_{kn+j} x^n$, then
the $k$-kernel of $(u_n)_{n\geq 0}$ is in bijection with the smallest 
set containing the series $\sum_{n=0}^{\infty}  u_n x^n$
that is stable under the operations of $\Lambda_j$ ($j=0,1,\ldots,k-1$).
We use a double list $L$ to encode the structure of the kernel, by $L[i][j]=i'$
we mean that the $i$-th element of the kernel is mapped to the $i'$-th by 
$\Lambda_j$, with the sequence itself denoted by the $0$-th element.

In this article we will consider the {\it Thue-Morse sequence}  
$\mathbf{t}=(t_n)$
defined by the recurrence relations
(see \cite{Thue1912}, \cite{Allouche1998Sh}) 
\begin{align*}
	t_0&=1; \\
	t_{2n}&= t_n; &(n\geq 1)\\
	t_{2n+1}&=-t_n, &(n\geq 0)
\end{align*}
and the {\it period-doubling sequence}  
$\mathbf{s}=(s_n)$
defined by the recurrence relations
\cite{Schaeffer2016, Fokkink2017etal}
\begin{align*}
	s_{2n}&= 1; &(n\geq 0)\\
	s_{2n+1}&=-s_n. &(n\geq 0)
\end{align*}
We see from the definition that the $2$-kernel of the Thue-Morse sequence
is $$\{(t_{n})_n, (t_{2n+1})_n\},$$
and the $2$-kernel of the period-doubling sequence is
$$\{(s_n)_n, (s_{2n})_n, (s_{2n+1})_n,(s_{4n+1})_n\}.$$
Therefore they are both $2$-automatic.
The structures of the above two 2-kernels are 
represented by $[[0,1], [1,0]]$ and $[[1,2],[1,1],[3,0], [3,3]]$ respectively.

\medskip

Basic definition and properties of continued fractions will be recalled in 
Section~\ref{sec:frac}.  
We consider the continued fractions defined by the Thue-Morse and the
period-doubling sequence:
\begin{equation}\label{eq:defc}
C(x):=\sum_{n\geq 0} c_n x^n 
:=
\cfrac {t_0}{ 1+ \cfrac{t_1x}{ 1+ \cfrac{t_2x}{ 1+\cfrac {t_3x}{ 1+\cfrac{t_4x}{\ddots}}}} }
= 
\cfrac {1}{ 1- \cfrac{x}{ 1- \cfrac{x}{ 1+\cfrac {x}{ 1-\cfrac{x}{\ddots}}}} }
\end{equation}
and 
\begin{equation}\label{eq:defd}
D(x):=\sum_{n\geq 0} d_n x^n 
:=
\cfrac {s_0}{ 1+ \cfrac{s_1x}{ 1+ \cfrac{s_2x}{ 1+\cfrac {s_3x}{ 1+\cfrac{s_4x}{\ddots}}}} }
= 
\cfrac {1}{ 1- \cfrac{x}{ 1+ \cfrac{x}{ 1+\cfrac {x}{ 1+\cfrac{x}{\ddots}}}} }
\end{equation}
The above two continued fractions will be called {\it Thue-Morse} continued fraction and {\it Period-doubling} continued fraction respectively.
Write $\bar{c}_n=\pi(c_n)$, $\bar{C}(x)=\sum_{n\geq 0} \bar{c}_n
x^n$, and $\bar{d}_n=\pi(d_n)$, $\bar{D}(x)=\sum_{n\geq 0} \bar{d}_n
x^n$, where $\pi$ is the canonical surjection of $\mathbb{Z}$ onto $\mathbb{Z}/
4\mathbb{Z}$. Note that we could also have defined $\bar{C}(x)$ and $\bar{D}(x)$
by \eqref{eq:defc} and \eqref{eq:defd} while viewing $\pm 1$ as elements
in $\mathbb{Z}/4\mathbb{Z}$.
The first terms of these sequences are listed below.
\begin{align*}
	(t_n)&=(1,-1,-1,1,-1,1,1,-1,-1,1,1,-1,1,-1,-1,1,\ldots),\\
	(c_n)&=(1, 1, 2, 3, 4, 6, 8, 11, 14, 18, 20, 22, 16, 4, -32, -93, -220, \ldots),\\
	(\bar{c}_n)&=( 1, 1, 2, 3, 0, 2, 0, 3, 2, 2, 0, 2, 0, 0, 0, 3, 0, 2, 0, 2, 0, 0, 0,  \ldots), \\
	(s_n)&=(1,-1,1,1,1,-1,1,-1,1,-1,1,1,1,-1,1,1,\ldots),\\
	(d_n)&=
(1, 1, 0, 1, -2, 4, -8, 17, -36, 74, -152, 316, -656, 1352,  \ldots),
\\
	(\bar{d}_n)&=
(1, 1, 0, 1, 2, 0, 0, 1, 0, 2, 0, 0, 0, 0, 0, 1, 2, 0, 0, 2, 0, 0, 0, 0,\ldots ).
\end{align*}
Notice that the sequences $(c_n), (\bar{c}_n), (d_n), (\bar{d}_n)$ are not in the OEIS.

\medskip

In the present paper we study the above two continued fractions and obtain 
the following properties of the sequences $(\bar{c}_n)$ and $(\bar{d}_n)$.

\begin{Theorem}\label{th:main} 
  We have the following congruence:
\begin{equation}\label{eq:mainf}
C(x)
  \equiv
  \frac{ \sqrt{1-4x}-1}{2x} + 1+ \sqrt{2\sqrt{1-4x}-1} \pmod 4.
\end{equation}
\end{Theorem}

\begin{Theorem}\label{th:mainpd}
  We have the following congruence:
\begin{equation}\label{eq:mainf-pd}
D(x)  
  \equiv
\frac{(1+\sqrt{1+4x})\sqrt{2\sqrt{1-4x^2}-1} -2}{2x}
  \pmod 4.
\end{equation}
\end{Theorem}

	The following Theorem from \cite{Denef1987} then allows us to conclude
	that $(\bar{c}_n)_n$ and $(\bar{d}_n)_n$ are $2$-automatic.
\begin{Theorem}[Denef-Lipschitz]\label{th:dl}
	Suppose that the power series $f(x_1,\ldots,x_k)\in\mathbb{Z}_p[[x_1,\ldots,x_k]]$
	is algebraic over $\mathbb{Z}_p[x_1,\ldots,x_k]$. Then for each $\alpha$, the 
	coefficient sequence of $f \pmod{p^\alpha}$ is $p$-automatic.
\end{Theorem}

The automaticity of  $(\bar{c}_n)_n$ and $(\bar{d}_n)_n$ can also be proved by
a direct calculation of their $2$-kernels.

\begin{Theorem}\label{th:mainmod4}
	The sequence $(\bar{c}_n)$ is $2$-automatic; the structure of its $2$-kernel
is represented by
$[[1, 2], [3, 4], [5, 6], [1, 7], [4, 7], [5, 4], [8, 6], [7, 7], [8, 4]]$.
\end{Theorem}

\begin{Theorem}\label{th:mainmod4pd}
	The sequence $(\bar{d}_n)$ is $2$-automatic; the structure of its $2$-kernel
	is represented by
	$[[1,0],[2,3],[1,4],[3,3],[4,3]]$
\end{Theorem}

The right hand side of congruence $\eqref{eq:mainf}$ and $\eqref{eq:mainf-pd}$
are respectively of degree $4$ and $8$ over $\mathbb{Z}(x)$. This raises the
question of what the minimal degree of polynomial equations 
that $\bar{C}$ and $\bar{D}$
satisfy is. Concerning this, we have the following result.

\begin{Theorem}\label{th:degree4}
	Let $S(x,y)=(xy^2+y+1)^2\in \mathbb{Z}/4\mathbb{Z}[x,y]$, then for both 
	series $\bar{C}(x)$ and $\bar{D}(x)$ in $\mathbb{Z}/4\mathbb{Z}$,
	we have $S(x,\bar{C}(x))=S(x,\bar{D}(x))=0$. Furthermore, there is no 
	polynomial in $\mathbb{Z}/4\mathbb{Z}[x,y]$ that, seen as a polynomial
	in $y$, has degree less than $4$, and, 
	 whose leading coefficient is invertible in the ring of Laurent series
	$\mathbb{Z}/4\mathbb{Z}((x))$, that annihilates either $\bar{C}(x)$ 
	or $\bar{D}(x)$.
\end{Theorem}

Informally put, Theorem \ref{th:degree4} says that $\bar{C}(x)$ and 
$\bar{D}(x)$ are of degree $4$, while their continued fraction expansion
are $2$-automatic. It may be interesting to compare Theorem \ref{th:degree4} to
the following result concerning automatic sequences and real continued 
fractions \cite{Bugeaud2013}.
\begin{Theorem}[Bugeaud 2013]
	The continued fraction expansion of an algebraic number of degree at least
	three cannot be generated by a finite automaton.
\end{Theorem}

The {\it Hankel determinant} of order $n$
of the formal power series $f(x)=a_0+a_1x+a_2x^2+\cdots$ 
(or of the sequence $(a_0,a_1,a_2, \cdots$)  is defined by 
\begin{equation*}
	H_n(f(x))=	H_n(a_0, a_1, a_2, \ldots):=\det (a_{i+j})_{0\leq i,j\leq n-1}
\end{equation*}
for $n\geq 1$, and  $H_0(f(x))=H_0(a_0, a_1, a_2, \ldots)=1$ if $n=0$.

Concerning the Hankel determinants of $C(x)$ and $D(x)$, we have the following
result.
\begin{Theorem}
	The sequences of Hankel determinants $(H_n({C}(x) ))$ and
	$(H_n({D}(x) ))$ are $2$-automatic.
\end{Theorem}

\medskip

Based on our results, we put forward the following conjecture.
\begin{Conjecture}\label{conj}
The sequences $c_n \pmod {2^m}$  and $d_n \pmod {2^m}$ are $2$-automatic
for all $m\geq 1$.
\end{Conjecture}

Theorem \ref{th:mainmod4} and \ref{th:mainmod4pd} says that Conjecture 
\ref{conj} is true for $m=2$.  Note that if the conjecture is true for
$m=k$, then it is also true for all positive integers $m<k$.  
For $m=1$, we can also see directly that 
$$
C(x)\equiv D(x)\equiv \cfrac{1}{1-\cfrac{x}{1-\cfrac{x}{1-\cfrac{x}{\ddots}}}}\pmod 2.$$
The right hand side of the congruence is the generating function for the
Catalan numbers \cite{Branden2002CS}. Being quadratic, it is $2$-automatic modulo $2$.

When $m=3$, experiments suggest that $c_n \pmod{2^3}$ and $d_n \pmod{2^3}$ are $2$-automatic 
with the following kernel structure for $c_n \pmod{2^3}$
\begin{gather*}
	[[1, 2], [3, 4], [5, 6], [7, 8], [9, 10], [11, 12], [13, 6], [3, 14], [8, 10], 
[4, 8],\\
	[10, 10], [11, 15], [12, 8], [16, 12], [17, 10], [15, 8], [16, 15], [14, 8]];
\end{gather*}
and for $d_n \pmod{2^3}$
\begin{gather*}
	[[1, 2], [3, 4], [5, 2], [6, 7], [4, 4], [8, 9], [3, 9], [10, 4], [11, 12], \\
	[9, 4], [7, 9], [8, 4], [13, 4], [12, 9]].
\end{gather*}

This article is structured as follows: in Section \ref{sec:frac}, we give the 
definitions and properties of Stieltjes and Jacobi continued fractions. In 
Section \ref{sec:tm}, we exploit the structure of the Thue-Morse sequence
and obtain the relations between certain subsequences of $P_n(x)$ and $Q_n(x)$, with $P_n(x)/Q_n(x)$ being the canonical representation of the $n$-th convergent
of the continued fraction $C(x)$.
Then we prove by induction the explicit expression of eight subsequences. 
We only use two of them but we need all eight for the induction hypotheses.
Taking the limit, we obtain the explicit expression of the Thue-Morse
continued fraction $\bar{C}(x)$ as a power series and prove that it is 
equal to an algebraic series with integer coefficients modulo $4$. In
consequence, its coefficients form a $2$-automatic 
sequence. In Section \ref{sec:pd} we obtain similar results for the period-doubling continued fraction $\bar{D}(x)$ using what we have proved for $\bar{C}(x)$
and the relation between the Thue-Morse and the period-doubling sequences. 
In Section \ref{sec:degree4} we prove Theorem \ref{th:degree4}.
Finally in Section \ref{sec:hankel} we prove that the sequences of Hankel
determinants $H_n(C(x))$ and $H_n(D(x))$ are $2$-automatic.


\section{Stieltjes and Jacobi continued fractions}\label{sec:frac}  
Stieltjes and Jacobi continued fraction
are studied in enumerative combinatorics for their link with 
the orthogonal polynomials and the
weighted
Motzkin paths (see \cite[p.386, p.389]{Jones1980T}, \cite{Viennot1983Ort}, \cite{Flajolet1980}).  
For a sequence $\mathbf{a}=(a_n)_n$ taking values in a field $\mathbb{K}$, and for each positive
integer $n$, we define the rational fraction:
\begin{equation}\label{eq:fr:Stieltjes}
	\Stiel_n(\mathbf{a}):=\cfrac {a_0}{ 1+ \cfrac{a_1x}{ 1
  +\cfrac{a_2x}{\quad \cfrac{\ddots}{1+  \cfrac{a_{n-1}x}{1+a_nx}}}}},
\end{equation}
which we also denote by $[\![a_0,a_1,\ldots,a_n]\!]$ for short.

We define two sequence of polynomials $P_n(x)$ and $Q_n(x)$ by the initial 
conditions $P_0(x)=a_0$, $Q_0(x)=1$, $P_1(x)=a_0$ and $Q_1(x)=1+a_1x$, 
and for $n\geq 2$
\begin{equation}\label{pqdef}
\begin{pmatrix}1& a_n x\\ 1 & 0\end{pmatrix}\cdots\begin{pmatrix}1&a_2 x \\ 1&0\
	\end{pmatrix}\begin{pmatrix}P_1(x)&Q_1(x)\\P_0(x)& Q_0(x) \end{pmatrix}=\begin{pmatrix}
		P_n(x) & Q_n(x) \\ P_{n-1}(x) & Q_{n-1}(x)\end{pmatrix}.
		\end{equation}
We have $\Stiel_n(\mathbf{a})=P_n(x)/Q_n(x)$ for all $n$. 
A proof of the following theorem can be found in \cite[p. 257]{foata2000principes}.
\begin{Theorem}\label{th:fr:convergence}
  The sequence of formal power series $P_n(x)/Q_n(x)$ is convergent.
\end{Theorem}
The infinite Stieltjes continued fraction $\Stiel(\mathbf{a})$ is defined to be
$$\lim\limits _{n\rightarrow \infty} P_n(x)/Q_n(x),$$ the rational fraction
$P_n(x)/Q_n(x)$
is called the {\it $n$-th convergent} of $\Stiel(\mathbf{a})$ and  the 
unsimplified fraction $P_n(x)/Q_n(x)$ the {\it canonical representation} of $\Stiel_n(\mathbf{a})$.

For $0\leq k < n$, if $P(x)/Q(x)$ is the canonical representation of the 
Stieltjes continued fraction $[\![a_k,\ldots,a_n]\!]$, then it can be easily shown from
\eqref{pqdef} that
\begin{align}\label{test}
	P_n(x)&=Q(x)P_{k-1}(x)+xP(x)P_{k-2}(x),\\
	Q_n(x)&=Q(x)Q_{k-1}(x)+xP(x)Q_{k-2}(x).\label{test2}
\end{align}

We define the Jacobi continued fractions in a similar way.
 For two sequences $\mathbf{u}=(u_n)_n$ and $\mathbf{v}=(v_n)_n$ with $v_i\neq 0$ for all 
 $i\in\mathbb{N}$, $\Jac(\mathbf{u},\mathbf{v})$  is defined to be
 the infinite continued fraction
\begin{equation}\label{eq:fr:Jacobi}
\Jac(\mathbf{u},\mathbf{v})
=
\cfrac {v_0 }{1 + u_1 x - 
  \cfrac{v_1 x^2 }{ 1+u_2x - 
    \cfrac{v_2 x^2 }{ {1 + u_3x - 
\cfrac{v_3 x^2}{ \ddots}}} }}. 
\end{equation}

The basic properties on Stieltjes and Jacobi continued fractions can be found in
\cite{Flajolet1980, Wall1948, Stieltjes1894Re, Heilermann1846}. We emphasize the fact that
the Hankel determinants can be calculated
from the Stieltjes and Jacobi continued fractions by means of the following  fundamental relation, first
stated by Heilermann in 1846 \cite{Heilermann1846}:
\begin{Theorem}\label{th:Hankel}
	The $n$th-order Hankel determinants of the Stieltjes \eqref{eq:fr:Stieltjes} and Jacobi \eqref{eq:fr:Jacobi} continued fractions are given by
	\begin{align*}
		H_n(\Stiel(\mathbf{a}))&=a_0^n(a_1a_2)^{n-1} (a_3a_4)^{n-2} \cdots (a_{2n-3} a_{2n-2}),\\
		H_n(\Jac(\mathbf{u},\mathbf{v})) 
		&= v_0^n v_1^{n-1} v_2^{n-2} \cdots v_{n-2}^2 v_{n-1}. 
\end{align*}
\end{Theorem}

The following contraction theorem establishes 
a link between the Stieltjes and Jacobi continued fractions
\cite{Wall1948, Perron1957II, Stieltjes1894Re}.

\begin{Theorem}\label{th:Contraction}[Contraction Theorem]
	The Stieltjes continued fraction $Stiel(\mathbf{a})$ and Jacobi continued
	fraction $Jac(\mathbf{u},\mathbf{v})$ are equal, if
	\begin{align*}
	u_1&=a_1; \\
		u_k&=a_{2k-2} + a_{2k-1}; &\text{($k\geq 2$)} \\
	v_0&=a_0; \\
		v_k&=a_{2k-1}  a_{2k}. &\text{($k\geq 1$)}
	\end{align*}
\end{Theorem}

Using the above notation, the two power series $C(x)$ and $D(x)$ defined in Section~\ref{sec:Intro} can be written as
$C(x)=\Stiel(\mathbf{t})$ and $D(x)=\Stiel(\mathbf{s})$.

\section{Thue-Morse continued fraction}\label{sec:tm} 

First we consider the $n$-th convergent $P_n(x)/Q_n(x)$ of the Thue-Morse 
continued fraction $C(x)$. Making use of the structure of the Thue-Morse
sequence, we establish the following recurrence relations of 
$P_n$ and $Q_n$.
\begin{Lemma}\label{th:pq}
	Let $P_n(x)/Q_n(x)$ be the canonical representation of $\Stiel_n(a)$. 
The two sequences $P_n(x)$ and $Q_n(x)$
are characterized by the initial conditions
$$P_0(x)=P_1(x)=Q_0(x)=1,\ Q_1(x)=1-x$$
  and
the following recurrence relations for $m\geq 1$ and $1\leq \epsilon\leq 2^m$:
\begin{align*}
  U_{2^{m+1}-\epsilon}(x)&=Q_{2^m-\epsilon}(-x)U_{2^m-1}(x)-x P_{2^m-\epsilon}(-x)U_{2^m-2}(x),
\end{align*}
where $U$ is either of the sequences $P$ or $Q$.
\end{Lemma}
\begin{proof}
	For a fixed $1\leq \epsilon\leq 2^m$, let $P(x)/Q(x)$ be the canonical 
	representation of the Stieltjes continued fraction 
	$[\![t_{2^m}, t_{2^m+1}, \ldots, t_{2^{m+1}-\epsilon}]\!]$.
From the definition of the Thue-Morse sequence, we see that $t_n=1$ if the
number of $1$'s in the binary expansion of $n$ is even, and $t_n=-1$ 
otherwise, and therefore
	$t_{2^m+j}=-t_j$ for all $m\geq 0$ and  $0\leq j\leq 2^m-1$. 
	Hence $P(x)/Q(x)$ is in fact the canonical representation of 
	$[\![-t_0,-t_1,\ldots,-t_{2^m-\epsilon}]\!]$. By \eqref{pqdef}, 
	$P(x)=-P_{2^m-\epsilon}(-x)$ and
	$Q(x)=Q_{2^m-\epsilon}(-x)$.
	Using formula \eqref{test} and \eqref{test2} we get the desired result.
\end{proof}

From the above recurrence relations of $P_n(x)$ and $Q_n(x)$, we are able to 
derive by induction the explicit expression of $P_{2^{2m}-2}(x)$ and
$Q_{2^{2m}-2}(x)$, which we will then use to calculate
$C(x)=\lim P_n(x)/Q_n(x)$.

To simplify notations, we define, for $m\geq 0$,
\begin{equation*}
	S_m(x)=\sum_{j=0}^{m-1} x^{2^j},\quad
	S^e_m(x)=\sum_{j=0}^{m-1} x^{2^{2j}},\quad
	S^o_m(x)=\sum_{j=0}^{m-1} x^{2^{2j+1}},
\end{equation*}
and 
\begin{equation*}
	S_\infty(x)=\sum_{j=0}^{\infty} x^{2^j},\quad
	S^e_\infty(x)=\sum_{j=0}^{\infty} x^{2^{2j}},\quad
	S^o_\infty(x)=\sum_{j=0}^{\infty} x^{2^{2j+1}}.
\end{equation*}
If the parameter is $x$, we write without the parameter as $S_m:=S_m(x)$, etc.
Recall that the Kronecker delta symbol $\delta_{i,j}$ is 1 if $i=j$, and 0 otherwise.

We are only interested in 3) and 7) from the following proposition, 
but we need the others for the proof by induction.
\begin{Proposition}\label{th:explicit}
We have the following explicit values for the polynomials $P_n(x)$ and $Q_n(x)$
for $n=2^k-1$ and $n=2^k-2$.\\

	1) $	P_{2^{2m}-1}(x)\equiv1+2 S^o_{m-1}(x)\pmod{4};  (m\geq 1)\label{eq:p1}$\\

	2) $	P_{2^{2m+1}-1}(x)\equiv1+2x (1-\delta_{m,0})+2S^e_m(x)\pmod{4}; (m\geq 0)\label{eq:p2}$\\

	3) $P_{2^{2m}-2}(x)\equiv   1 + x^{-1} S_{2m-1}(x)^2 - 2 S^e_m(x)\pmod{4}; 
	                       (m\geq 1)\label{eq:p3}$\\

	4) $P_{2^{2m+1}-2}(x)\equiv   1 + x^{-1} S_{2m}(x)^2 - 2 S^o_m(x)\pmod{4}; 
												 (m\geq 0)\label{eq:p4}$\\

	5) $Q_{2^{2m}-1}(x)\equiv   1 -x + 2 x^{2^{2m-1}}
		 - S_{2m-1}(x)^2 + 2x S^e_m(x)\pmod{4};(m\geq 1) \label{eq:q1}$\\
		 
	6) $Q_{2^{2m+1}-1}(x)\equiv 1 -x+   2 x^{2^{2m}} (1-\delta_{m,0})		 
		 - S_{2m}(x)^2 + 2x S^o_m(x)\pmod{4};(m\geq 0) \label{eq:q2}$\\

	7) $Q_{2^{2m}-2}(x)\equiv1+2 S_{2m-1}(x)\pmod{4}; (m\geq 1)\label{eq:q3}$\\

	8) $Q_{2^{2m+1}-2}(x)\equiv1+2x(1-\delta_{m,0})+2 S_{2m}(x)\pmod{4}. ( m\geq 0)$\label{eq:q4}
\end{Proposition}

\begin{proof}
We prove this result by induction on 
	$$n\in  \{2^k-1 \mid k\} \cup \{ 2^k-2 \mid k\}.$$
	When we compute $P_n(x)$ or $Q_n(x)$, the induction hypothesis
	is that the expressions for $P_\ell(x)$ and $Q_\ell(x)$ are true for $\ell <n$ and
	$\ell\in  \{2^k-1 \mid k\} \cup \{ 2^k-2 \mid k\}$.
	Relations 1) - 8) are true for $m\equiv 0$ or $m\equiv 1$. In the sequel let $m\geq 2$.

	1)  Using the induction hypothesis, we have
\begin{align*}
	Q_{2^{2m-1}-1}(-x)	  
	&\equiv 1 +x+   2 x^{2^{2m-2}} - S_{2m-2}(x)^2 - 2x S^o_{m-1}(x);\\
	P_{2^{2m-1}-1}(x)& \equiv 1+2x +2S^e_{m-1}(x);\\
	P_{2^{2m-1}-1}(-x)&\equiv 1-2x +2S^e_{m-1}(x);\\
	P_{2^{2m-1}-2}(x)&\equiv    1 + x^{-1} S_{2m-2}(x)^2 - 2 S^o_{m-1}(x).
\end{align*}
By Lemma \ref{th:pq}, we obtain
\begin{align*}
	P_{2^{2m}-1}&\equiv Q_{2^{2m-1}-1}(-x)P_{2^{2m-1}-1}(x)-x P_{2^{2m-1}-1}(-x)P_{2^{2m-1}-2}(x)\\
	&\equiv  	(1+2x +2S^e_{m-1}(x))  \\
		&  \quad \times \bigl( (1 +x+   2 x^{2^{2m-2}} - S_{2m-2}(x)^2 - 2x S^o_{m-1}(x))\\
	& \qquad -x 
	(   1 + x^{-1} S_{2m-2}(x)^2 - 2 S^o_{m-1}(x)) \bigr)\\
	&\equiv  	(1+2x +2S^e_{m-1}(x))  
		 ( 1 +   2 x^{2^{2m-2}}  - 2S_{2m-2}(x)^2  )\\
	&\equiv  	1+2x +2S^e_{m-1}(x) +   2 x^{2^{2m-2}}  - 2S_{2m-2}(x)^2  \\
	&\equiv  	1 +2S^e_{m-1}(x) +   2 x^{2^{2m-2}}  - 2S_{2m-1}(x)  \\
	&\equiv  	1 +2S^o_{m-1}(x).
\end{align*}
	2)  Using the induction hypothesis, we have
\begin{align*}
	Q_{2^{2m}-1}(-x)	 
	&\equiv    1 +x + 2 x^{2^{2m-1}} - S_{2m-1}(x)^2 + 2x S^e_m(x); \\
	P_{2^{2m}-1}(x)& \equiv 1+2S^o_{m-1}(x);\\
	P_{2^{2m}-1}(-x)&\equiv 1+2S^o_{m-1}(x);\\
	P_{2^{2m}-2}(x)&\equiv    1 + x^{-1} S_{2m-1}(x)^2 - 2 S^e_m(x).
\end{align*}

By Lemma \ref{th:pq}, we obtain
\begin{align*}
	P_{2^{2m+1}-1}&\equiv Q_{2^{2m}-1}(-x)P_{2^{2m}-1}(x)-x P_{2^{2m}-1}(-x)P_{2^{2m}-2}(x)\\
	&\equiv  	(1+2S^o_{m-1}(x))  \\
		&  \quad \times \bigl( (1 +x+   2 x^{2^{2m-1}} - S_{2m-1}(x)^2 + 2x S^e_{m}(x))\\
	& \qquad -x 
	(   1 + x^{-1} S_{2m-1}(x)^2 - 2 S^e_{m}(x)) \bigr)\\
	&\equiv  	(1+2S^o_{m-1}(x))  
		 ( 1 +   2 x^{2^{2m-1}}  - 2S_{2m-1}(x)^2  )\\
	&\equiv  	1 +2S^o_{m-1}(x) +   2 x^{2^{2m-1}}  - 2S_{2m-1}(x)^2  \\
	&\equiv  	1 +2S^o_{m-1}(x) +   2 x^{2^{2m-1}}  - 2S_{2m}(x) +2x \\
	&\equiv  	1 +2x+  2S^e_{m}(x).
\end{align*}
	3)  Using the induction hypothesis, we have
\begin{align*}
	Q_{2^{2m-1}-2}(x)&\equiv 1+2x+2 S_{2m-2};\\
	P_{2^{2m-1}-1}(x)& \equiv 1+2x +2S^e_{m-1};\\
	P_{2^{2m-1}-2}(x)&\equiv    1 + x^{-1} S_{2m-2}^2 - 2 S^o_{m-1};\\
	P_{2^{2m-1}-2}(-x)&\equiv    1 - x^{-1} S_{2m-2}^2 - 2 S^o_{m-1}.
\end{align*}
By Lemma \ref{th:pq}, we obtain
\begin{align*}
P_{2^{2m}-2}(x)
&\equiv Q_{2^{2m-1}-2}(-x)P_{2^{2m-1}-1}(x)
		 -x P_{2^{2m-1}-2}(-x)P_{2^{2m-1}-2}(x) \\
&\equiv  (1+2x+2 S_{2m-2}) (1+2x +2S^e_{m-1})\\
&\quad -x ( 1 + x^{-1} S_{2m-2}^2 - 2 S^o_{m-1})
 	 (   1 - x^{-1} S_{2m-2}^2 - 2 S^o_{m-1})\\
&\equiv  1+2 S_{2m-2} + 2S^e_{m-1}
	 -x \bigl( 1 - 2 S^o_{m-1} \bigr)^2
	 +x \bigl(  x^{-1} S_{2m-2}^2  \bigr)^2\\
&\equiv  1+2 S_{2m-2} + 2S^e_{m-1} -x +x^{-1}   (S_{2m-1} -x )^2\\
	&\equiv  1 +  x^{-1} S_{2m-1}^2+ 2S^e_{m}. 
\end{align*}
4)  Using the induction hypothesis, we have
\begin{align*}
	P_{2^{2m}-1}(x)&\equiv 1+2 S^o_{m-1}(x);\\
	P_{2^{2m}-2}(x)&\equiv    1 + x^{-1} S_{2m-1}(x)^2 - 2 S^e_m(x); \\
	Q_{2^{2m}-2}(x)&\equiv 1+2 S_{2m-1}(x).
\end{align*}
By Lemma \ref{th:pq}, we obtain
\begin{align*}
	P_{2^{2m+1}-2}(x)
	&\equiv Q_{2^{2m}-2}(-x)P_{2^{2m}-1}(x)
	-xP_{2^{2m}-2}(-x)P_{2^{2m}-2}(x)\\
	&\equiv (1+2S_{2m-1}(-x))(1+2S^o_{m-1}(x))\\
	&\quad -x(1-x^{-1}S_{2m-1}(-x)^2-2S^e_m(-x))(1+x^{-1}S_{2m-1}(x)^2-2S^e_m(x))\\
	&\equiv 1+2S_{2m-1}(x)+2S^o_{m-1}(x)-x+x^{-1}S_{2m-1}(x)^4\\
	&\equiv 1+2S_{2m-1}(x)+2S^o_{m-1}(x)-x+x^{-1}(S_{2m}(x)^2-2xS_{2m}(x)+x^2)\\
	&\equiv 1+x^{-1}S_{2m}(x)^2-2S^o_m(x).
\end{align*}
5) Using the induction hypothesis, we have
\begin{align*}
	P_{2^{2m-1}-1}(x)&\equiv 1+2x +2S^e_{m-1}(x);\\
	Q_{2^{2m-1}-1}(x)&\equiv  1 -x+   2 x^{2^{2m-2}} 		 - S_{2m-2}(x)^2 + 2x S^o_{m-1}(x);\\
	Q_{2^{2m-1}-2}(x)&\equiv 1+2x+2 S_{2m-2}(x).
\end{align*}
By Lemma \ref{th:pq}, we obtain
\begin{align*}
	Q_{2^{2m}-1}(x)&\equiv Q_{2^{2m-1}-1}(-x)Q_{2^{2m-1}-1}(x)-xP_{2^{2m-1}-1}(-x)Q_{2^{2m-1}-2}(x)\\
	&\equiv (1+x+2x^{2^{2m-2}}-S_{2m-2}^2+2xS^o_{m-1})\\
	&\quad (1-x+2x^{2^{2m-2}}-S_{2m-2}^2+2xS^o_{m-1})\\
	&\quad -x(1+2x+2S^e_{m-1})(1+2x+2S_{2m-2})\\
	&\equiv (1-S^2_{2m-2})^2-x^2-(x+2xS^e_{m-1}+2xS_{2m-2})\\
	&\equiv 1+2S^2_{2m-2}+S^2_{2m-1}+2xS_{2m-1}-(x+2xS^e_{m-1}+2xS_{2m-2})\\
	&\equiv 1-x+2x^{2^{2m-1}}-S^2_{2m-1}+2xS^e_m.
\end{align*}
6) Using the induction hypothesis, we have
\begin{align*}
	P_{2^{2m}-1}(x)&\equiv 1+2 S^o_{m-1}(x);\\
	Q_{2^{2m}-1}(x)&\equiv    1 -x + 2 x^{2^{2m-1}}
		 - S_{2m-1}(x)^2 + 2x S^e_m(x);\\
	Q_{2^{2m}-2}(x)&\equiv 1+2 S_{2m-1}(x).
\end{align*}
By Lemma \ref{th:pq}, we obtain
\begin{align*}
	Q_{2^{2m+1}-1}(x)&\equiv Q_{2^{2m}-1}(-x)Q_{2^{2m}-1}(x)-x P_{2^{2m}-1}(-x)Q_{2^{2m}-2}(x)\\
	&\equiv (1+x+2x^{2^{2m-1}}-S_{2m-1}^2+2xS_m^e)\\
	&\equiv (1-x+2x^{2^{2m-1}}-S_{2m-1}^2+2xS^e_m)\\
	&\quad -x(1+2S^o_{m-1})(1+2S_{2m+1})\\ 
	&\equiv 1+3S^2_{2m}+2xS_{2m}+2x^{2^{2m}}-x+2xS^o_{m-1}+2xS_{2m-1}\\
	&\equiv 1-x+2x^{2^{2m}}-S_{2m}^2+2xS^o_m.
\end{align*}
7) Using the induction hypothesis, we have
\begin{align*}
	P_{2^{2m-1}-2}(x)&\equiv    1 + x^{-1} S_{2m-2}(x)^2 - 2 S^o_{m-1}(x); \\
	Q_{2^{2m-1}-1}(x)&\equiv  1 -x+   2 x^{2^{2m-2}}
		 - S_{2m-2}(x)^2 + 2x S^o_{m-1}(x);\\
	Q_{2^{2m-1}-2}(x)&\equiv 1+2x+2 S_{2m-2}(x).
\end{align*}
By Lemma \ref{th:pq}, we obtain
\begin{align*}
	&\quad	Q_{2^{2m}-2}(x)\\
&\equiv Q_{2^{2m-1}-2}(-x)Q_{2^{2m-1}-1}(x)-xP_{2^{2m-1}-2}(-x)Q_{2^{2m-1}-2}(x)\\
	&\equiv (1+2x+2S_{2m-2})(1-x+2x^{2^{2m-2}}-S_{2m-2}^2+2xS^o_{m-1})\\
	&\quad -x(1-x^{-1}S_{2m-2}^2+2S^o_{m-1})(1+2x+2S_{2m-2})\\
	&\equiv 1+x+2x{2^{2m-2}}-S_{2m-2}^2+2xS^o_{m-1}+2x^2+2xS_{2m-2}^2+2S_{2m-2}+2xS_{2m-2}\\
	&\quad+2S_{2m-2}^3
	 -x(1+2x+2S_{2m-2}-x^{-1}S^2_{2m-2}+2S^2_{2m-2}+2x^{-1}S^3_{2m-2}+2S^o_{m-1})\\
	&\equiv 1+2S_{2m-1}.
\end{align*}
8) Using the induction hypothesis, we have
\begin{align*}
	P_{2^{2m}-2}(x)&\equiv    1 + x^{-1} S_{2m-1}(x)^2 - 2 S^e_m(x); \\
	Q_{2^{2m}-1}(x)&\equiv    1 -x + 2 x^{2^{2m-1}}
		 - S_{2m-1}(x)^2 + 2x S^e_m(x); \\
	Q_{2^{2m}-2}(x)&\equiv 1+2 S_{2m-1}(x); 
\end{align*}
By Lemma \ref{th:pq}, we obtain
\begin{align*}
	Q_{2^{2m+1}-2}(x)&\equiv Q_{2^{2m}-2}(-x)Q_{2^{2m}-1}(x)-xP_{2^{2m}-2}(-x)Q_{2^{2m}-2}(x)\\
	&\equiv (1+2S_{2m-1})(1-x+2x^{2^{2m-1}}-S_{2m-1}^2+2xS^e_m)\\
	&\quad -x(1-x^{-1}S^2_{2m-1}+2S^e_m)(1+2S_{2m-1})\\
	&\equiv 1-x+2x^{2^{2m-1}}-S^2_{2m-1}+2xS^e_m+2S_{2m-1}+2xS_{2m-1}+2S^3_{2m-1}\\
	&\quad -(x-S_{2m-1}^2+2xS^e_m +2xS_{2m-1}+2S^3_{2m-1})\\
	&\equiv 1+2x+2S_{2m-1}+2x^{2^{2m-1}}\\
	&\equiv 1+2x+2S_{2m}.\qedhere
\end{align*}

\end{proof}

The explicit expressions of $P_{2^{2m}-2}(x)$ and $P_{2^{2m}-2}(x)$ give the
explicit expression for $C(x)$.
 
\begin{Proposition}\label{prop:c}
	\begin{equation}\label{eq:cexplicit}
  C(x) \equiv 1-  \sum_{i,j=0}^\infty x^{2^i+2^j-1} +2 \sum_{k=0}^\infty x^{2^{2k}}
		\pmod{4}.
	\end{equation}
\end{Proposition}
\begin{proof}
By Theorem \ref{th:fr:convergence}, 
$$C(x)=\lim_{m\rightarrow \infty}P_{2^{2m}-2}(x)/Q_{2^{2m}-2}(x).$$
	The constant term of $Q_{2^{2m}+2}(x)$ being $1$, 
	$1/Q_{2^{2m}+2}(x)$ belongs to $\mathbb{Z}[[x]]$. By 3) and 7) of Proposition
	\ref{th:explicit},
\begin{align*}
C(x)
  &\equiv
\lim_{m \rightarrow \infty} \frac{ 1 + x^{-1} S_{2m-1}(x)^2 - 2 S^e_m(x)}{ 1+2 S_{2m-1}(x) } \\
  &\equiv
  \lim_{m \rightarrow \infty} ( 1 + x^{-1} S_{2m-1}(x)^2 - 2 S^e_m(x))( 1+2 S_{2m-1}(x) ) \\
  &\equiv
  \lim_{m \rightarrow \infty}  1 + x^{-1} S_{2m-1}(x)^2 - 2 S^e_m(x)
  + 2 S_{2m-1}(x)   + 2x^{-1} S_{2m-1}(x)^3   \\
	&\equiv 1+x^{-1}(\sum\limits_{j=0}^\infty x^{2^j})^2+2\sum\limits_{j=0}^\infty
	x^{2^{2j}}+2\sum\limits_{j=0}^\infty x^{2^j}+2x^{-1}(\sum\limits_{j=0}^\infty x^{2^j})^3\\
	&\equiv 1+x^{-1}(\sum\limits_{j=0}^\infty x^{2^j})^2+2\sum\limits_{j=0}^\infty
	x^{2^{2j}}+2\sum\limits_{j=0}^\infty x^{2^j}+2x^{-1}(\sum\limits_{j=1}^\infty
	x^{2^j})(\sum\limits_{j=0}^\infty x^{2^j})\\
	&\equiv 1+x^{-1}(\sum\limits_{j=0}^\infty x^{2^j})^2+2\sum\limits_{j=0}^\infty
	x^{2^{2j}}+2\sum\limits_{j=0}^\infty x^{2^j}+2x^{-1}(\sum\limits_{j=0}^\infty
	x^{2^j})^2+2\sum\limits_{j=0}^\infty x^{2^j}\\
	&\equiv 1-x^{-1}(\sum\limits_{j=0}^\infty x^{2^j})^2+2\sum\limits_{j=0}^\infty
	x^{2^{2j}}\\
	&\equiv 1-  \sum_{i,j=0}^\infty x^{2^i+2^j-1} +2 \sum_{k=0}^\infty x^{2^{2k}}
	\pmod{4}.\qedhere
\end{align*}
\end{proof}

Now we prove Theorem \ref{th:mainmod4} by repeated applying the Cartier 
operators to the right hand side of \eqref{eq:cexplicit}.

\begin{proof}[Proof of Theorem \ref{th:mainmod4}]
	We recall that $\bar{C}(x)$ denotes the series in $\mathbb{Z}/4\mathbb{Z}[[x]]$
	that is the reduction modulo $4$ of $C(x)$.
	We prove that by applying $\Lambda_0$ and $\Lambda_1$ repeatedly to $\bar{C}(x)$, 
	we can only obtain a finite number of series. Indeed, we have
	$$\bar{C}(x)=:f_0,$$
	$$\Lambda_0 f_0=\Lambda_0 \bar{C}(x)=1+2\sum\limits_{j=0}^{\infty}x^{2^j}+2\sum\limits_{k=0}^
	{\infty}x^{2^{2k+1}}=1+2\sum\limits_{j=0}^\infty x^{2^{2j}}=:f_1,$$
	$$\Lambda_1f_0=\Lambda_1 \bar{C}(x)=-1-\sum\limits_{i,j=0}^\infty x^{2^i+2^j-1}+2
	=1-\sum\limits_{i,j=0}^\infty x^{2^i+2^j-1}=:f_2,$$
	$$\Lambda_0f_1=\Lambda_0\Bigl(1+2\sum\limits_{j=0}^\infty x^{2^{2j}}\Bigr)=1+2\sum\limits_{j=0}^\infty x^{2^{2j+1}}=:f_3,$$
	$$\Lambda_1f_1=\Lambda_1\Bigl(1+2\sum\limits_{j=0}^\infty x^{2^{2j}}\Bigr)=2=:f_4,$$
	$$\Lambda_0f_2=\Lambda_0\Bigl(1-\sum\limits_{i,j=0}^\infty x^{2^i+2^j-1}\Bigr)=1+2\sum\limits_{j=0}^\infty x^{2^j}=:f_5,$$
	$$\Lambda_1f_2=\Lambda_1\Bigl(1-\sum\limits_{i,j=0}^\infty x^{2^i+2^j-1}\Bigr)=-\sum\limits_{i,j=0}^\infty x^{2^i+2^j-1}=:f_6,$$
	$$\Lambda_0f_3=\Lambda_0\Bigl(1+2\sum\limits_{j=0}^\infty x^{2^{2j+1}}\Bigr)=1+2\sum\limits_{j=0}^\infty x^{2^{2j}}=f_1,$$
	$$\Lambda_1f_3=\Lambda_1\Bigl(1+2\sum\limits_{j=0}^\infty x^{2^{2j+1}}\Bigr)=0=:f_7,$$
	$$\Lambda_0f_4=\Lambda_{0}2=2=f_4,$$
	$$\Lambda_1f_4=\Lambda_{1}2=0=f_7,$$
	$$\Lambda_0f_5=\Lambda_0\Bigl(1+2\sum\limits_{j=0}^\infty x^{2^j}\Bigr)=f_5,$$
	$$\Lambda_1f_5=\Lambda_1\Bigl(1+2\sum\limits_{j=0}^\infty x^{2^j}\Bigr)=2=f_4,$$
	$$\Lambda_0f_6=\Lambda_0\Bigl(-\sum\limits_{i,j=0}^\infty x^{2^i+2^j-1}\Bigr)=-1+2\sum\limits_{j=0}^\infty x^{2^j}=:f_8,$$
	$$\Lambda_1f_6=\Lambda_1\Bigl(-\sum\limits_{i,j=0}^\infty x^{2^i+2^j-1}\Bigr)=-\sum\limits_{i,j=0}^\infty x^{2^i+2^j-1}=f_6,$$
	$$\Lambda_0f_8=\Lambda_0\Bigl(-1+2\sum\limits_{j=0}^\infty x^{2^j}\Bigr)=f_8,$$
	$$\Lambda_1f_8=\Lambda_0\Bigl(-1+2\sum\limits_{j=0}^\infty x^{2^j}\Bigr)=2=f_4.$$
	We see from the computation above that the $2$-kernel of $\bar{C}(x)$ consists
	of $9$ elements, $f_0$ through $f_8$. The structure of the $2$-kernel is 
	$$[[1, 2], [3, 4], [5, 6], [1, 7], [4, 7], [5, 4], [8, 6], [7, 7], [8, 4]].\qedhere
$$
\end{proof}
 The following lemma  is used in the proof
 of Theorem \ref{th:main} (see, for example, \cite{han2015hankel}).
 \begin{Lemma}\label{th:lem}
	 $$\sqrt{1-4x}\equiv 1+2\sum\limits_{k=1}^{\infty} x^{2^k} \pmod{4}.$$
 \end{Lemma}
\begin{proof}[Proof of Theorem \ref{th:main}]
From the proof of Proposition \ref{prop:c}  we know that
	$$C(x)\equiv1+x^{-1}S_\infty(x)^2+2S^e_\infty(x)+2S_\infty(x)+2x^{-1}S_\infty(x)^3 \pmod{4},$$
therefore, we only need to find 
	$S_\infty(x) \pmod{2}$, $(S_\infty (x))^2\pmod{4}$ and  $S_\infty ^e (x) \pmod{2}$. 
	By Lemma \ref{th:lem}, 
	\begin{equation}\label{eq:ss}
	S_\infty(x)\equiv\frac{1-\sqrt{1-4x}}{2}\pmod{2},\end{equation}
	so that
	\begin{equation}\label{eq:S2}
	S_\infty(x)^2\equiv \left(\frac{1-\sqrt{1-4x}}{2}\right)^2 \equiv
	\frac{1-2x-\sqrt{1-4x}}{2} \pmod{4}.
\end{equation}
	To calculate $S^e_\infty(x)\pmod{2}$, we notice that
	\begin{align*}
		S_\infty^e(x)^2+S^e_\infty(x)&=S_\infty^e(x^2)+S^e_\infty(x)+2x\psi(x)\\
		&=S_\infty(x)+2x\psi(x)\\
		&=\frac{1-\sqrt{1-4x}}{2}+2x\xi(x)+2x\psi(x),
	\end{align*}
	where 
	$$\psi(x)=\frac{1}{2x}\left(S_\infty^e(x)^2-S_\infty^e(x^2)\right) \text{\quad and\quad}
	\xi(x)=\frac{1}{2x}\left(S_\infty(x)-\frac{1-\sqrt{1-4x}}{2}\right)$$ 
	are in $\mathbb{Z}[[x]]$.
	We remark that by Lemma \ref{th:lem},
if 
	$$f(x),g(x)\in \mathbb{Z}[[x]] \text{\quad and\quad } f(x)\equiv g(x)\pmod{2},$$
	then
$$\sqrt{1+4xf(x)}\equiv \sqrt{1+4xg(x)}\pmod{4}.$$
	Therefore 
	\begin{align}
		S_\infty^e(x)&=\frac{-1+\sqrt{1-(2-2\sqrt{1-4x})+2x\xi(x)+2x\psi(x)}}{2}\nonumber \\
		&	\equiv \frac{-1+\sqrt{2\sqrt{1-4x}-1}}{2} \pmod{2}.\label{eq:Se} 
	\end{align}
Finally
	\begin{align*}
		C(x)&\equiv1+x^{-1}S_\infty(x)^2+2S^e_\infty(x)+2S_\infty(x)+2x^{-1}S_\infty(x)^3 \\
		&\equiv 1+\frac{1-2x-\sqrt{1-4x}}{2x}+\left(-1+\sqrt{2\sqrt{1-4x}-1}\right)\\\
		&\quad+ \left(1-\sqrt{1-4x}\right)+2x^{-1}\left(\frac{1-\sqrt{1-4x}}{2}\right)^3\\
		&\equiv \frac{ \sqrt{1-4x}-1}{2x} + 1+ \sqrt{2\sqrt{1-4x}-1} \pmod 4.\qedhere
	\end{align*}
\end{proof}


\section{Period-doubling continued fraction}\label{sec:pd} 

In this section we prove Theorem \ref{th:mainpd} using Theorem \ref{th:main}
and Theorem \ref{th:Contraction}.
As a corollary, we get the explicit expression of $\bar{D}(x)$ as a power series
and from this we calculate the $2$-kernel of the sequence $(\bar{d}_n)$.

	\begin{proof}[Proof of Theorem \ref{th:mainpd}]
		
In Theorem \ref{th:Contraction}, if we let
  \begin{align*}
  u_1&=t_1=-1; \\
  u_n&=t_{2n-2} + t_{2n-1} = 0; \qquad\hbox{($n\geq 2$)} \\
  v_0&=t_0 = 1; \\
    v_n&=t_{2n-1}  t_{2n} = -t_{n-1} t_n = s_{n-1}, \qquad\hbox{($n\geq 1$)}
  \end{align*}
	we get
\begin{equation}\label{eq:cd}
  C(x)=\cfrac {t_0}{ 1+ \cfrac{t_1x}{ 1+ \cfrac{t_2x}{ 1+\cfrac {t_3x}{ 1+\cfrac{t_4x}{\ddots}}}} }
=\cfrac {1 }{1 - x - 
  \cfrac{s_0 x^2 }{ 1 - 
    \cfrac{s_1 x^2 }{ {1  - 
\cfrac{s_2 x^2}{ \ddots}}} }}
  =\frac {1}{1-x-x^2 D(-x^2)}.
\end{equation}
		We define
\begin{align*}
  H_1(x) &= {\frac{ \sqrt{1-4x}-1}{2x} + 1+ \sqrt{2\sqrt{1-4x}-1}}
  =1-3x+\cdots \\
  H_2(x) &= \frac{1+\sqrt{1+4x}}{2}=1+x +\cdots \\
  H_3(x) &= \sqrt{2\sqrt{1-4x^2}-1}  =1-2x^2 +\cdots
\end{align*}
		Then our goal \eqref{eq:mainf-pd} can be written as
		$$D(x)\equiv \frac{H_1(x)H_3(x)-1}{x} \pmod{4}.$$
		Since $C(x)\equiv H(x) \pmod{4}$ and the constant term of $C(x)$ and $H_1(x)$ is $1$, by \eqref{eq:cd} we know that
		$$
			-x^2D(-x^2)=\frac{1}{C(x)}-(1-x) \equiv \frac{1}{H_1(x)}-1+x \pmod{4}.
			$$
	We only need to show that
	$$\frac{1}{H_1(x)}-1+x\equiv -x^2\times\frac{H_2(-x^2)H_3(-x^2)-1}{-x^2}\pmod{4},$$
that is, 
		$$\frac{1}{H_1(x)}+x\equiv H_2(-x^2) H_3(-x^2) \pmod{4}.$$
Since the constant term of $H_1(x)$ is $1$, this is equivalent to
		$$H_1(x)(H_2(-x^2)H_3(-x^2)-x)\equiv 1 \pmod{4}.$$
By \eqref{eq:S2}, \eqref{eq:Se} and \eqref{eq:ss},
\begin{align*}
	H_1(x)&\equiv -\frac{S_\infty(x)^2}{x}+2S_\infty^e+1 \pmod{4},\\
	H_2(-x^2)&= \frac{1+\sqrt{1-4x^2}}{2}\\
	&\equiv 1-x^2-S_\infty(x^2)^2\\
	&\equiv 1-x^2-(S_\infty(x)-x)^2\\
	&\equiv 1+(2x-1)S_\infty(x)^2 \pmod{4},\\
	H_3(-x^2) &=\sqrt{2\sqrt{1-4x^4}-1}\\
	&\equiv 1+2S_\infty^e(x^4)\\
	&\equiv 1-2x+2S_\infty^e(x) \pmod{4}.
\end{align*}
	Taking account of the above congruence relations and after rearranging the
	terms, we get
\begin{align}
	&	xH_1(x)(H_2(-x^2)H_3(-x^2)-x)\label{eq:hh}\\
	\equiv&\ x-xS_\infty(x)^2+x^2-S_\infty(x)^2
	(1-S_\infty(x)^2+x)\nonumber \\
	&\qquad +S_\infty^e(x)(2x^2+2S_\infty(x)^2(1-S_\infty(x)^2)) \pmod{4}\nonumber.
\end{align}
		Since by \eqref{eq:ss} and \eqref{eq:S2} we have
\begin{align*}
	2S_\infty(x)^2(1-S_\infty(x)^2)&\equiv 2(S_\infty(x)-x)(1-S_\infty(x)+x)\\
	&\equiv 2(S_\infty(x)-S_\infty(x)^2-x-x^2)\\
	&\equiv 2x^2 \pmod{4},
\end{align*}
and
\begin{align*}
	S_\infty(x)^4&\equiv (S_\infty(x)-x)^2\\
	&\equiv S_\infty(x)^2+x^2-2xS_\infty(x)\pmod{4},
\end{align*}
congruence \eqref{eq:hh} becomes
\begin{align*}
  & xH_1(x)(H_2(-x^2)H_3(-x^2)-x)\label{eq:hh}\\
	\equiv & x-xS_\infty(x)^2+x^2-S_\infty(x)^2
  (1-S_\infty(x)^2+x) \\
\equiv & x+2x^2+2xS_\infty(x)^2+2xS_\infty(x)\\
	\equiv & x \pmod{4}.\qedhere
\end{align*}
	\end{proof}

From Theorem \ref{th:mainpd} we obtain the following explicit expression for
$D(x) \pmod{4}$.
\begin{Corollary}\label{th:dcoro}
\begin{equation}
  D(x) \equiv 1-  \sum_{i,j=0}^\infty x^{2^i+2^j-1} 
  +2\sum_{k=0}^\infty x^{2^{2k+1}-1}\Bigl(1+ \sum_{j=0}^\infty x^{2^{j}}\Bigr)
	\pmod{4}.
\end{equation}
\end{Corollary}
\begin{proof}[Proof of Corollary \ref{th:dcoro}]
	 We obtained the following congruence in Theorem \ref{th:mainpd}, 
	$$D(x)\equiv\frac{(1+\sqrt{1+4x})\sqrt{2\sqrt{1-4x^2}-1}-2}{2x}\pmod{4}.$$
	From \eqref{eq:S2} we know 
	\begin{align*}
			\frac{1+\sqrt{1+4x}}{2}
		&\equiv 1+x-S_\infty(-x)^2 \pmod{4}\\
		&\equiv 1+x-S_\infty(x)^2 \pmod{4}\\
		&\equiv 1+x-\Bigl(\sum\limits_{j=0}^\infty x^{2^j}\Bigr)^2\pmod{4}\\
		&\equiv 1+x-\sum\limits_{i,j=0}^\infty x^{2^i+2^j}\pmod{4}.
	\end{align*}
	By \eqref{eq:Se}
	$$\sqrt{2\sqrt{1-4x^2}-1}\equiv 1+2\sum\limits_{k=0}^\infty x^{2^{2k+1}} 
	\pmod{4}.$$
	Therefore
	\begin{align*}
		D(x)&\equiv \frac{1}{x}\Bigl(\Bigl(1+x-\sum\limits_{i,j=0}^
		\infty x^{2^i+2^j}\Bigr)\Bigl(1+2\sum\limits _{k=0}^\infty x^{2^{2k+1}} 
		\Bigr)-1	\Bigr)\\
		&\equiv 1-\sum\limits_{i,j=0}^ \infty x^{2^i+2^j-1}+
		2\sum\limits _{k=0}^\infty x^{2^{2k+1}-1}\Bigl(1+x-\sum\limits_{i,j=0}^
    \infty x^{2^i+2^j}\Bigr)\\
		&\equiv 1-\sum\limits_{i,j=0}^ \infty x^{2^i+2^j-1}+
		2\sum\limits _{k=0}^\infty x^{2^{2k+1}-1}\Bigl(1+x-\sum\limits_{j=0}^
		\infty x^{2^{j+1}}\Bigr)\\
		&\equiv 1-\sum\limits_{i,j=0}^ \infty x^{2^i+2^j-1}+
		2\sum\limits _{k=0}^\infty x^{2^{2k+1}-1}\Bigl(1+\sum\limits_{j=0}^
		\infty x^{2^{j}}\Bigr). \qedhere
	\end{align*}
\end{proof}

\begin{proof}[Proof of Theorem \ref{th:mainmod4pd}]
	By Theorem \ref{th:mainpd} and Theorem \ref{th:dl} we know that the sequence
	$(\bar{d}_n)$ is $2$-automatic.
	Using Corollary \ref{th:dcoro}, we calculate
	 the $2$-kernel of $(\bar{d}_n)$.
	First we compute $\Lambda_0(\bar{D}(x))$ and $\Lambda_1(\bar{D}(x))$. 
	We define three power series in $\mathbb{Z}/4\mathbb{Z}[[x]]$:
	$$
		A:= -  \sum_{i,j=0}^\infty x^{2^i+2^j-1},\quad
		B:= 2\sum_{k=0}^\infty x^{2^{2k+1}-1},\quad
		C:=2 \Bigl(\sum_{k=0}^\infty x^{2^{2k+1}-1}\Bigr) \Bigl( \sum_{j=0}^\infty x^{2^{j}}\Bigr), 
	$$
	so that $\bar{D}(x)=1+A+B+C$.
	We have $$\Lambda_0(A)=\Lambda_0(2\sum\limits_{j=1}^{\infty}x^{2^j})=
	\Lambda_0(2\sum\limits_{j=0}^{\infty}(x^2)^{2^j})= 
	2\sum\limits_{j=0}^{\infty}x^{2^j},$$
	\begin{align*}
		\Lambda_1(A)&=\Lambda_1\Bigl(-x-\sum\limits_{i,j =1 }^{\infty} x^{2^i+2^j-1}
	\Bigr)\\
		&=\Lambda_1\Bigl(x\Bigl(-1-\sum\limits_{i,j =1 }^{\infty} (x^2)^{2^{i-1}+2^{j-1}-1}\Bigr) \Bigr)\\
		&=-1-\sum\limits_{i,j=0}^{\infty}x^{2^i+2^j-1},
	\end{align*}
	$$\Lambda_0(B)=0,$$
	$$\Lambda_1(B)=\Lambda_1\Bigl(2x\sum\limits_{k=0}^{\infty}x^{2^{2k+1}-2}
	\Bigr)=\Lambda_1\Bigl(2x\sum\limits_{k=0}^{\infty}(x^2)^{2^{2k}-1}
  \Bigr)= 2\sum\limits_{k=0}^{\infty}x^{2^{2k}-1},$$
	$$\Lambda_0(C)=\Lambda_0\Bigl(2\sum\limits_{k=0}^{\infty} x^{2^{2k+1}}\Bigr)
	=\Lambda_0\Bigl(2\sum\limits_{k=0}^{\infty} (x^2)^{2^{2k}}\Bigr)
	=2\sum_{k=0}^{\infty}x^{2^{2k}},$$
	\begin{align*}
		\Lambda_1(C)&=\Lambda_1\Bigl(2x\sum\limits_{k=0,j=1}^{\infty}x^{2^{2k+1}+
	2^j-2}\Bigr)\\
		&=\Lambda_1\Bigl(2x\sum\limits_{k=0,j=1}^{\infty}(x^2)^{2^{2k}+
	2^{j-1}-1}\Bigr)\\
		&=2\sum\limits_{k,j=0}^{\infty}x^{2^{2k}+
  2^{j}-1}.
	\end{align*}
	Thus, if we let $f_0$ denote $\bar{D}(x)$, then 
	\begin{align*}
		\Lambda_0(f_0)=\Lambda_0(\bar{D}(x))&=1+\Lambda_0(A)+\Lambda_0(B)+\Lambda_0(C)\\
		&=1+2\sum\limits_{j=0}^{\infty}x^{2^j}+2\sum_{k=0}^{\infty}x^{2^{2k}}\\
		&=1+2\sum\limits_{j=0}^{\infty}x^{2^{2j+1}}\\
		&=:f_1,
	\end{align*}
	and
	\begin{align*}
		\Lambda_1(f_0)=\Lambda_1(\bar{D}(x))&=\Lambda_1(A)+\Lambda_1(B)+\Lambda_1(C)\\
		&=-1-\sum\limits_{i,j=0}^{\infty}x^{2^i+2^j-1}+ 2\sum\limits_{k=0}^{\infty}x^{2^{2k}-1}+2\sum\limits_{k,j=0}^{\infty}x^{2^{2k}+
  2^{j}-1}\\
		&=\bar{D}(x)=f_0.
	\end{align*}
		The last equality holds because
		\begin{align*}
			&\bar{D}(x)-\Bigl(-1-\sum\limits_{i,j=0}^{\infty}x^{2^i+2^j-1}+ 2\sum\limits_
			{k=0}^{\infty}x^{2^{2k}-1}+2\sum\limits_{k,j=0}^{\infty}x^{2^{2k}+
			2^{j}-1}\Bigr)\\
			&=2+2\sum_{k=0}^\infty x^{2^{2k+1}-1}\Bigl(1+\sum\limits_{j=1}^{\infty}x^{2^j}\Bigr)
			+2\sum_{k=0}^\infty x^{2^{2k}-1}\Bigl(1+\sum\limits_{j=1}^{\infty}x^{2^j}\Bigr)\\
			&=2+2\sum\limits_{k=0}^{\infty}x^{2^k-1}\Bigl(1+\sum\limits_{j=1}^{\infty}x^{2^j}\Bigr)\\
			&=2+2\sum\limits_{k=0}^{\infty}x^{2^k-1}+2\sum\limits_{j,k=0}^{\infty}
			x^{2^k+2^j-1}\\
				&=2+2\sum\limits_{k=0}^{\infty}x^{2^k-1}+2\sum\limits_{j=0}^{\infty}
				x^{2^j+2^j-1}\\
				&=2+2\sum\limits_{k=0}^{\infty}x^{2^k-1}+2\sum\limits_{j=0}^{\infty}
				x^{2^{j+1}-1}\\
				&=2+2x^{2^0-1}\\
				&=0.
	\end{align*}

Then we calculate $\Lambda_0(f_1)$ and 
$\Lambda_1(f_1)$:
$$\Lambda_0(f_1)=\Lambda_0\Bigl( 1+2\sum\limits_{j=0}^{\infty}
x^{2^{2j+1}}\Bigr)=\Lambda_0\Bigl( 1+2\sum\limits_{j=0}^{\infty}
(x^2)^{2^{2j}}\Bigr)=1+2\sum\limits_{j=0}^\infty x^{2^{2j}}=:f_2,$$
$$\Lambda_1(f_1)=\Lambda_1\Bigl( 1+2\sum\limits_{j=0}^{\infty}
x^{2^{2j+1}}\Bigr)=0=:f_3.$$

Finally we calculate $\Lambda_0(f_2)$ and
$\Lambda_1(f_2)$:
\begin{align*}
	\Lambda_0(f_2)&=\Lambda_0\Bigl(
1+2\sum\limits_{j=1}^\infty x^{2^{2j}}\Bigr)\\
	&=\Lambda_0\Bigl(
1+2\sum\limits_{j=0}^\infty (x^2)^{2^{2j+1}}\Bigr)\\
	&=1+2\sum\limits_{j=0}^\infty x^{2^{2j+1}}\\
	&=\Lambda_0(\bar{D}(x))=f_1,
\end{align*}
$$\Lambda_1(f_2)=\Lambda_1(2x)=2=:f_4.\qedhere$$
We see that the structure of the $2$-kernel of $(\bar{d}_n)$ is
	$[[1,0],[2,3],[1,4],[3,3],[4,3]]$.
\end{proof}
\section{Proof of Theorem \ref{th:degree4}}\label{sec:degree4}
In this section we prove Theorem \ref{th:degree4}.  
First we recall that from Theorem \ref{th:main} and \ref{th:mainpd} that
$$C(x)\equiv \varphi(x) \pmod{4},$$
$$D(x)\equiv \psi(x) \pmod{4},$$
where
$$\varphi(x)= \frac{ \sqrt{1-4x}-1}{2x} + 1+ \sqrt{2\sqrt{1-4x}-1}\in\mathbb{Z}
[[x]],$$
$$\psi(x)= \frac{(1+\sqrt{1+4x})\sqrt{2\sqrt{1-4x^2}-1} -2}{2x}\in \mathbb{Z}
[[x]].$$
By rearranging the terms and squaring both sides of the equalities, 
we obtain annihilating polynomials $P(x,y)$ and $Q(x,y)$ of $\varphi(x)$ 
and $\psi(x)$ respectively:
\begin{align*}P(x,y) =&y^{4} x^{2} - 4 \, y^{3} x^{2} + 2 \, y^{3} x^{1} + 8 \, y^{2} x^{2} - 4 \, y^{2} x^{1} + 8 \, y x^{2} \\&+ 16 \, x^{3} + y^{2} - 16 \, x^{2} + 8 \, x^{1} - 1,\\
	Q(x,y)=&	y^{8} x^{7} + 8 \, y^{7} x^{6} + 4 \, y^{6} x^{6} + 30 \, y^{6} x^{5} + 32 \, y^{4} x^{7} + 24 \, y^{5} x^{5}\\& + 64 \, y^{4} x^{6} + 68 \, y^{5} x^{4} + 14 \, y^{4} x^{5} + 128 \, y^{3} x^{6} + 48 \, y^{4} x^{4}\\& + 256 \, y^{3} x^{5} + 64 \, y^{2} x^{6} + 97 \, y^{4} x^{3} + 56 \, y^{3} x^{4} + 224 \, y^{2} x^{5}\\& + 256 \, x^{7} + 32 \, y^{3} x^{3} + 372 \, y^{2} x^{4} + 128 \, y x^{5} + 84 \, y^{3} x^{2} + 78 \, y^{2} x^{3}\\& + 192 \, y x^{4} - 96 \, x^{5} - 12 \, y^{2} x^{2} + 232 \, y x^{3} + 64 \, x^{4} + 40 \, y^{2} x\\& + 44 \, y x^{2} + 73 \, x^{3} - 24 \, y x + 52 \, x^{2} + 8 \, y + 8 \, x - 8.
\end{align*}
For $n\in \mathbb{N}^*$, we let $\pi_n$ denote the canonical projection of 
$\mathbb{Z}$ onto $\mathbb{Z}/n\mathbb{Z}$, and by abuse of notation, 
the canonical projection of $\mathbb{Z}[[x]]$ onto $\mathbb{Z}/n\mathbb{Z}
[[x]]$, of $\mathbb{Z}[x,y]$ onto $\mathbb{Z}/n\mathbb{Z}[x,y]$, etc.

Since $P(x,\varphi(x))=0$, $Q(x,\psi(x))=0$, and
$$\pi_2(P(x,y))=x^2y^4+y^2+1=(xy^2+y+1)^2,$$
$$\pi_2(Q(x,y))=x^7y^8+x^3y^4+x^3=x^3(xy^2+y+1)^4,$$
we have
\begin{equation}\label{eq:cmod2}
x\varphi(x)^2+\varphi(x)+1\equiv 0\pmod{2},
\end{equation}
\begin{equation}\label{eq:dmod2}
x\psi(x)^2+\psi(x)+1\equiv 0\pmod{2},
\end{equation}
and therefore
$$(x\varphi(x)^2+\varphi(x)+1)^2 \equiv 0\pmod{4},$$
$$(x\psi(x)^2+\psi(x)+1)^2 \equiv 0\pmod{4}.$$
In other words, the polynomial 
$S(x,y)=(xy^2+y+1)^2\in \mathbb{Z}/4\mathbb{Z}[x,y]$ is
an annihilating polynomial for both $\bar{C}=\pi_4(\varphi)$ and $\bar{D}=
\pi_4(\psi)$.

Now we prove that there is no 
	polynomial in $\mathbb{Z}/4\mathbb{Z}[x,y]$ that, seen as a polynomial
	in $y$, has degree less than $4$, and, 
	 whose leading coefficient is invertible in the ring of Laurent series
	$\mathbb{Z}/4\mathbb{Z}((x))$, that annihilates either $\bar{C}(x)$ 
	or $\bar{D}(x)$. By absurdity, suppose that $Q(x,y)=Q_n(x)y^n+\cdots+
	Q_1(x)y+Q_0(x)$ is such a polynomial of minimal degree on $y$.
	By assumption, $n$ is less than $4$, $Q_n(x)$ is invertible 
	in $\mathbb{Z}/4\mathbb{Z}((x))$ and $Q(x,y)$
	annihilates either $\bar{C}(x)$ or $\bar{D}(x)$. Since $Q_n(x)$ is invertible
	in $\mathbb{Z}/4\mathbb{Z}((x))$, we can effectuate Euclidean division 
	of $P(x,y)$ by $Q(x,y)$, and by minimality of $n$, we obtain 
	$$Q_n(x)P(x,y)=Q(x,y)R(x,y)$$
	for some $R(x,y)\in \mathbb{Z}/4\mathbb{Z}[x,y]$.

	Reducing modulo $2$ (where we use $\pi_2$ by abuse of notation), we get
	$$\pi_2(Q(x,y))\pi_2(R(x,y))=\pi_2(Q_n(x)P(x,y))=
	\pi_2(Q_n(x))(xy^2+y+1)^2.$$
  Since $Q_n(x)$ is invertible in $\mathbb{Z}/4\mathbb{Z}((x))$, $\pi_2(Q
	(x))$ is non-zero.
	As factorization into irreducible factors of $\pi_2(Q(x,y)R(x,y))$
	in $\mathbb{F}_2(x)[y]$ is unique up to multiplication by elements in
	$\mathbb{F}_2(x)$, and $1\leq n\leq 3$, we know that 
	there exists $\alpha(x)\in\mathbb{Z}[x]$ taking coefficients in $\{0,1\}$, 
	such that $\pi_2(\alpha(x))$ is a factor of $\pi_2(Q_n(x))$ and
  $$\pi_2(Q(x,y))=\pi_2(\alpha(x))\cdot(xy^2+y+1).$$
	Therefore there exist polynomials $\beta_0(x), \beta_1(x),\beta_2(x)$
  in $\mathbb{Z}[x]$ taking coefficients in $\{0,1\}$, such that
	$$Q(x,y)=\pi_4(\alpha(x))\cdot(xy^2+y+1)+2x\pi_4(\beta_2(x))y^2
	+2\pi_4(\beta_1(x))y+ 2\pi_4(\beta_0(x)).$$
	Since, by assumption, $Q(x,\pi_4(f(x)))=0$, where $f$ stands for one of 
	$C$ and $D$, we have
  $$\alpha(x)(xf(x)^2+f(x)+1)\equiv 2x\beta_2(x)f(x)^2+2\beta_1(x)f(x)+2\beta_0
	(x) \pmod{4}.$$
	We let $g(x)$ denote the series $(xf(x)^2+f(x)+1)/2$, by \eqref{eq:cmod2} and
	\eqref{eq:dmod2} we know that $g(x)$ has integer coefficients.
	We rewrite the above congruence as
	$$
	\alpha(x)g(x)\equiv x\beta_2(x)f(x)^2+\beta_1(x)f(x)+\beta_0(x) \pmod{2},$$
	in other words,
	\begin{equation}\label{eq:42}
	\pi_2(\alpha(x))\pi_2(g(x))=x\pi_2(\beta(x))\pi_2(f(x))^2+\pi_2(\beta_1(x))
	\pi_2(f(x))+\pi_2(\beta_0(x)).
	\end{equation}

In light of \eqref{eq:cmod2} and \eqref{eq:dmod2}, $\pi_2(f(x))$ is of degree
$2$ over $\mathbb{F}_2(x)$, so that the right
	hand side of \eqref{eq:42} lives in a quadratic extension of 
	$\mathbb{F}_2 (x)$. 
	We will prove that the left hand side of \eqref{eq:42} is of degree
	$4$ over $\mathbb{F}_2(x)$, which will lead to a contradiction.
   Also, $\pi_2(\alpha(x))$ being a non-zero
	element in $\mathbb{F}_2(x)$, we only need to prove that the degree of 
	$\pi_2(g(x))$ over $\mathbb{F}_2(x)$ is $4$.

	In case $f(x)=C(x)$, since $C(x)\equiv \varphi(x) \pmod{4}$, we have
	$$xC(x)^2+C(x)+1\equiv x\varphi(x)^2+\varphi(x)+1 \pmod{4},$$
	and therefore 
	$$g(x)\equiv (x\varphi(x)^2+\varphi(x)+1 )/2\pmod{2}.$$ 
 From Theorem \ref{th:main}, we find that $(x\varphi(x)^2+\varphi(x)+1 )/2$
 is equal to
$$
x \sqrt{-4 \, x + 1} + x \sqrt{2 \, \sqrt{-4 \, x + 1} - 1} + \frac{1}{2} \, \sqrt{-4 \, x + 1} \sqrt{2 \, \sqrt{-4 \, x + 1} - 1} + \frac{1}{2} \, \sqrt{-4 \, x + 1}
$$
Its annihilating polynomial is
\begin{align*}
	T(x,y)=& 16 \, x^{6} + 8 \, y^{2} x^{3} + 32 \, y x^{4} - 32 \, x^{5} + y^{4} + 24 \, y^{2} x^{2} - 40 \, y x^{3} + 8 \, x^{4} \\&- 6 \, y^{2} x - 8 \, y x^{2} + 16 \, x^{3} + 8 \, y x - 8 \, x^{2} - y + x.
\end{align*}
Therefore $\pi_2(T(x,y))=y^4+y+x$ is an annihilating polynomial of 
$\pi_2(g(x))$. Let us verify that it is irreducible in $\mathbb{F}_2[x][y]$.
If $y^4+y+x$ factorizes into a cubic and a linear factor, then the linear
factor must be $(y+x)$ or $(y+1)$. However, $y^4+y+x$ is divisible by neither.
If it factorizes into two quadratic factors, then it must be of the form
$(y^2+\xi (x) y+1)(y^2+\eta(x)y+x)$, where $\xi(x)$ and $\eta(x)$ are in 
$\mathbb{F}_2[x]$.
When we expand and compare the coefficients,
we see that $\xi(x)$ and $\eta(x)$ must satisfy simultaneously
$\xi(x)+\eta(x)=0$ and $\xi(x)\eta(x)+x+1=0$, which is impossible.

In case $f(x)=D(x)$, we find that 
$$g(x)\equiv (x\psi(x)^2+\psi(x)+1)/2 \pmod{2}.$$
We could have computed an annihilating polynomial for $\pi_2(g(x))$
the same way that we did in the case $f(x)=C(x)$, but we would have to deal with
too many terms in the calculation involving
$D(x)$. So we choose to work directly in $\mathbb{F}_2[[x]]$, by
using Corollary \ref{th:dcoro}
to find the $2$-kernel of $\pi_2(g(x))$, from which we will obtain the minimal 
polynomial of $\pi_2(g(x))$ following the method in the proof of Theorem~1 from 
\cite{Christol1980KMFR}.

We prove now that the structure of the $2$-kernel of $\pi_2(g(x))$ is 
$$[[1, 0], [2, 3], [1, 4], [3, 4], [4, 4]].$$
By Corollary \ref{th:dcoro},
  \begin{equation}
  D(x) \equiv 1-  \sum_{i,j=0}^\infty x^{2^i+2^j-1}
  +2\sum_{k=0}^\infty x^{2^{2k+1}-1}\Bigl(1+ \sum_{j=0}^\infty x^{2^{j}}\Bigr)
		\pmod{4} .
\end{equation}
    Therefore
    $$D(x)\equiv 1+ \sum_{i,j=0}^\infty x^{2^i+2^j-1}
     \equiv 1+ \sum_{j=1}^\infty x^{2^j-1}
     \equiv \sum_{j=0}^\infty x^{2^j-1} \pmod{2},
    $$
    $$  xD(x)^2 \equiv  \sum_{i,j=0}^\infty x^{2^i+2^j-1} \pmod{4}, $$
		and
    $$xD(x)^2+D(x)+1\equiv 2+2\sum_{k=0}^\infty x^{2^{2k+1}-1}
    \Bigl(1+ \sum_{j=0}^\infty x^{2^{j}}\Bigr) \pmod{4}.$$
		We now have an explicit expression for $\pi_2(g(x))$
		$$\pi_2(g(x))= 1+1\sum_{k=0}^\infty x^{2^{2k+1}-1}
		\Bigl(1+ \sum_{j=0}^\infty x^{2^{j}}\Bigr)=:g_0(x) .$$
		To compute the $2$-kernel of $g_0(x)=\pi_2(g(x))$, we apply the operators
		$\Lambda_0$ and $\Lambda_1$:
		$$\Lambda_0{g_0(x)}=1+\sum_{k=0}^\infty x^{2^{2k}}=:g_1(x),$$
		$$\Lambda_1{g_0(x)}=g_0(x),$$
		$$\Lambda_0(g_1(x))=1+\sum_{k=0}^\infty x^{2^{2k+1}}=:g_2(x),$$
		$$\Lambda_1{g_1(x)}=1=:g_3(x),$$
		$$\Lambda_0{g_2(x)}=g_1(x),$$
		$$\Lambda_1{g_2(x)}=0=:g_4(x).$$
Therefore the $2$-kernel of $g$ is
$$[[1, 0], [2, 3], [1, 4], [3, 4], [4, 4]].\qedhere $$
The following identities are just another way of writing out $2$-kernel.
$$g_0(x)=g_1(x)^2+xg_0(x)^2,$$
$$g_1(x)=g_2(x)^2+xg_3(x)^2=g_2(x)^2+x,$$
$$g_2(x)=g_1(x)^2+xg_4(x)^2=g_1(x)^2.$$
From these, we deduce that $\pi_2(g(x))=g_0(x)$ is a root of 
the polynomial $$x^4y^8+y^4+xy^2+y+x^2$$ 
in $\mathbb{F}_2[x,y]$, which factorizes as 
$$(xy^4+y^3+1)(x^3y^4+x^2y^3+xy^2+y+x^2).$$
By computing the first few terms of $\pi_2(g(x))$ we find that the second factor
is not an annihilating polynomial for $\pi_2(g(x))$, and therefore 
$xy^4+y^3+1$ is. As 
$$xy^4+y^3+1=y^4\left(\left(\frac{1}{y}\right)^4+\frac{1}{y}+x\right),$$
and we have just shown in the case $f=C(x)$ that
$y^4+y+x$ is irreducible in $\mathbb{F}_2[x,y]$, $xy^4+y^3+1$ is also 
irreducible. This shows that $\pi_2(g(x))$ has degree $4$ over $\mathbb{F}_2(x)$
and completes our proof.

\section{Hankel determinants} \label{sec:hankel}

The Hankel determinants of ${C}(x)$ and ${D}(x)$
can be calculated by Heilermann's theorem. 
To prove their automacity, we need the following theorem.

\begin{Theorem}[see \cite{Allouche1986} ]\label{th:aq}

	Let $X$ be an alphabet on which is defined an associative operation $*$.
	Let $x=(x_n)$ be a $q$-automatic sequence on the alphabet $X$. The sequence
	$y=(y_n)$ defined by 
	\begin{align*}
		y_1&=x_0\\
		y_2&=x_1*x_0\\
		& \vdots \\
		y_n&=x_{n-1}*x_{n-2}*\cdots * x_0
	\end{align*}
	is $q$-automatic.
\end{Theorem} 
From the definition of the Thue-Morse and the period-doubling sequence,
it is easy to see that 
$$
t_{2k+1} t_{2k+2} = s_{k}, \text{\quad and \quad}
s_{2k+1} s_{2k+2} = -s_{k}.
$$
By Theorem \ref{th:Hankel}, we have
\begin{align*}
	H_n({C}(x))& = t_0^n (t_1t_2)^{n-1} (t_3t_4)^{n-2} \cdots (t_{2n-3} t_{2n-2})^1\\
	&= s_0^{n-1} s_1^{n-2} \cdots s_{n-2}^1.
\end{align*}
and
\begin{align*}
	H_n({D}(x)) &=s_0^n (s_1s_2)^{n-1} (s_3s_4)^{n-2} \cdots (s_{2n-3} s_{2n-2})^1\\
	& =(-1)^{n(n-1)/2} s_0^{n-1} s_1^{n-2} \cdots s_{n-2}^1\\ 
	&=(-1)^{n(n-1)/2} H_n(C(x)). 
\end{align*}

\medskip

Define $u_n:=s_0s_1\cdots s_{n-1}$. By Theorem \ref{th:aq}, $(u_n)$ is 
$2$-automatic, and consequently 
$H_n({C}(x))=u_0u_1\cdots u_{n-1}$ is $2$-automatic.
Since $(-1)^{n(n-1)/2}$ is periodic, $H_n({D}(x))$ is also
$2$-automatic. Finally $H_n(\bar{C}(x))$ and $H_n(\bar{C}(x))$ are $2$-automatic
as the reduction modulo $4$ of $H_n(C(x))$ and $H_n(D(x))$.

\medskip

{\bf Acknowledgments} The authors would like to thank Jean-Paul Allouche for
valuable suggestions and Zhi-Ying Wen for invitation to Tsinghua University
which facilitated collaboration.
\bibliographystyle{plain}

\bibliography{frtm}

\end{document}